\documentclass[reqno]{amsart}\usepackage{amsmath}

\usepackage{amsfonts}
\usepackage{amssymb}
\usepackage{hyperref}

\begin{document}
\title[ Multiple positive solutions ]{ Multiple positive solutions for a nonlinear
three-point integral boundary-value problem}
\author[F. Haddouchi, S. Benaicha]{Faouzi Haddouchi, Slimane Benaicha}
\address{Faouzi Haddouchi\\
Department of Physics, University of Sciences and Technology of
Oran, El Mnaouar, BP 1505, 31000 Oran, Algeria}
\email{fhaddouchi@gmail.com}
\address{Slimane Benaicha \\
Department of Mathematics, University of Oran, Es-senia, 31000 Oran,
Algeria} \email{slimanebenaicha@yahoo.fr}

\begin{abstract}

We investigate the existence of positive solutions to the nonlinear
second-order three-point integral boundary value problem
\begin{equation*} \label{eq-1}
\begin{gathered}
{u^{\prime \prime }}(t)+f(t, u(t))=0,\ 0<t<T, \\
u(0)={\beta}u(\eta),\ u(T)={\alpha}\int_{0}^{\eta}u(s)ds,
\end{gathered}
\end{equation*}
where $0<{\eta}<T$, $0<{\alpha}< \frac{2T}{{\eta}^{2}}$,
$0<{\beta}<\frac{2T-\alpha\eta^{2}}{\alpha\eta^{2}-2\eta+2T}$ are
given constants. We establish the existence of at least three positive solutions by using the
Leggett–-Williams fixed-point theorem.
\end{abstract}

\maketitle
\numberwithin{equation}{section}
\newtheorem{theorem}{Theorem}[section]
\newtheorem{lemma}[theorem]{Lemma}
\newtheorem{definition}[theorem]{Definition}
\newtheorem{example}[theorem]{Example}
\newtheorem{remark}[theorem]{Remark}
\allowdisplaybreaks

\section{Introduction}
The study of the existence of solutions of multipoint boundary value
problems for linear second-order ordinary differential equations was
initiated by II'in and Moiseev \cite{Mois}. Then Gupta \cite{Gupt}
studied three-point boundary value problems for nonlinear
second-order ordinary differential equations. Since then, nonlinear
second-order three-point boundary value problems have also been
studied by several authors. We refer the reader to \cite{Chen, Guo, Good1, Good2, Good3, Good4, Good5,
Han, Liang1, Liang2, Li, Liu1, Liu2, Liu3, Ma1, Ma2, Ma3, Pang,
Sun1, Sun2, Tarib, Webb1, Webb2, Webb3, Webb4, Xu, Yang1, Yang2, Zhang} and the references therein.

This paper is a continuation of our study in \cite{Had1,Had2} and is concerned with the existence and multiplicity of positive solutions of the problem

\begin{equation} \label{eq-4}
{u^{\prime \prime }}(t)+f(t, u(t))=0,\ t\in(0,T),
\end{equation}
with the three-point integral boundary condition
\begin{equation} \label{eq-5}
u(0)={\beta}u(\eta),\ u(T)={\alpha}\int_{0}^{\eta}u(s)ds,
\end{equation}

Throughout this paper, we assume the following hypotheses:

\begin{itemize}
\item[(H1)]
 $f\in C([0, T]\times[0,\infty),[0,\infty))$ and $f(t, .)$ does not
vanish identically on any subset of $[0, T]$ with positive measure.

\item[(H2)] $\eta\in(0, T)$, $0<{\alpha}< \frac{2T}{{\eta}^{2}}$ and
$0<{\beta}<\frac{2T-\alpha\eta^{2}}{\alpha\eta^{2}-2\eta+2T}$.
\end{itemize}

In this paper, by using the Leggett–-Williams fixed-point theorem \cite{Leg}, we will show the existence of at least three positive solutions for a three-point integral boundary value
problem. Some papers in this area include  \cite{Luo, Ping, Xian, Ander1, Ander2, Avery, He, Agar}.

\section{Background and definitions}

The proof of our main result is based on the Leggett–-Williams fixed point
theorem, which deals with fixed points of a cone-preserving operator
defined on an ordered Banach space. For the convenience of the reader, we
present here the necessary definitions from cone theory in Banach spaces.

\begin{definition}
Let $E$ be a real Banach space. A nonempty closed convex set
$P\subset E $ is called a cone if it satisfies the following two conditions:\\
(i) $x\in P$, $\lambda\geq0$ implies $\lambda x\in P$;\\
(ii) $x\in P$, $-x\in P$ implies $x=0$.

Every cone $P\subset E$ induces an ordering in $E$ given by
$x\leq y $ if and only if $y-x \in E$.
\end{definition}

\begin{definition}
An operator is called completely continuous if it is continuous and
maps bounded sets into precompact sets.
\end{definition}

\begin{definition}
A map $\psi$ is said to be a nonnegative continuous concave functional
on a cone $P$ of a real Banach space $E$ if $\psi : P\rightarrow [0, \infty)$ is continuous and
\[\psi(tx+(1-t)y)\geq t\psi(x)+(1-t)\psi(y)\]
for all $x, y \in P$ and $t\in [0, 1]$. Similarly we say the map $\varphi$ is a nonnegative
continuous convex functional on a cone $P$ of a real Banach space $E$ if $\varphi : P\rightarrow [0, \infty)$ is continuous and
\[\varphi(tx+(1-t)y) \leq t\varphi(x)+(1-t)\varphi(y)\]
for all $x, y \in P$ and $t\in [0, 1]$.
\end{definition}

\begin{definition}
Let $\psi$  be a nonnegative continuous concave functional on the cone $P$. Define the convex sets $P_{c}$
and $P(\psi, a, b)$ by

\[P_{c}=\{x\in P: \|x\|<c\},\ \text{for}\ \ c>0\]

\[P(\psi, a, b)=\{x\in P: a\leq\psi(x), \|x\|\leq b\},\ \text{for}\ \ 0<a<b. \]
\end{definition}

Next we state the Leggett–-Williams fixed-point theorem.

\begin{theorem} [\cite{Leg}]\label{theo 1.1}

Let $A : \overline{P}_{c}\rightarrow \overline{P}_{c}$ be a completely continuous operator and let $\psi$ be a nonnegative continuous concave functional on $P$ such that $\|\psi(x)\|\leq \|x\|$ for all $x\in \overline{P}_{c}$ . Suppose that there exist $0 < a < b < d \leq c $ such that the following conditions hold,

\begin{itemize}
\item[(C1)] $\{x \in P(\psi, b, d):  \psi(x) > b\}\neq\emptyset \  \text{and} \ \psi(Ax) > b \ \text{for all}\ x \in P(\psi, b, d)$;

\item[(C2)] $\|Ax\| < a \  \text{for}\ \|x\|\leq a$;

\item[(C3)] $\psi(Ax) > b\ \text{for}\ x \in P(\psi, b, c)\  \text{with}\ \|Ax\| > d$.
\end{itemize}
Then $A$ has at least three fixed points $x_{1}$, $x_{2}$ and $x_{3}$ in $\overline{P}_{c}$ satisfying

$\|x_{1}\| < a$,  $\psi(x_{2}) > b$,  $a < \|x_{3}\|$ with  $\psi(x_{3})< b$.

\end{theorem}

\section{Some preliminary results}
In order to prove our main result, we need some preliminary results. Let
us consider the following boundary value problem

\begin{equation}\label{eq-14}
{u^{\prime \prime }}(t)+y(t)=0, \ t\in (0,T),
\end{equation}
\begin{equation}\label{eq-15}
u(0)=\beta u(\eta ), \ u(T)=\alpha \int_{0}^{\eta }u(s)ds
\end{equation}

For problem \eqref{eq-14}-\eqref{eq-15}, we have the following conclusions which are derived
from \cite{Had1}.

\begin{lemma}[See \cite{Had1}]\label{lem 3.1}
Let $\beta \neq \frac{2T-\alpha \eta ^{2}}{\alpha \eta ^{2}-2\eta
+2T}$. Then for $y\in C([0,T],\mathbb{R})$, the problem \eqref{eq-14}-\eqref{eq-15} has
the unique solution
\begin{eqnarray*}
u(t)&=&\frac{\beta (2T-\alpha \eta ^{2})-2\beta (1-\alpha \eta
)t}{(\alpha \eta ^{2}-2T)-\beta (2\eta -\alpha \eta
^{2}-2T)}\int_{0}^{\eta }(\eta -s)y(s)ds \\
&&+\frac{\alpha \beta
\eta -\alpha (\beta -1)t}{(\alpha \eta ^{2}-2T)-\beta (2\eta -\alpha
\eta
^{2}-2T)}\int_{0}^{\eta }(\eta -s)^{2}y(s)ds \\
&&+\frac{2(\beta-1)t-2\beta \eta }{(\alpha \eta
^{2}-2T)-\beta (2\eta -\alpha \eta ^{2}-2T)}\int_{0}^{T}(T-s)y(s)ds-%
\int_{0}^{t}(t-s)y(s)ds.
\end{eqnarray*}
\end{lemma}

\begin{lemma}[See \cite{Had1}]\label{lem 3.2}
Let $0<\alpha <\frac{2T}{\eta ^{2}}$, $\ 0\leq \beta <%
\frac{2T-\alpha \eta ^{2}}{\alpha \eta ^{2}-2\eta +2T}$. If $y\in C([0,T],[0,\infty
))$, then the unique solution $u$ of problem
\eqref{eq-14}, \eqref{eq-15} satisfies $\ u(t)\geq 0$ for $t\in [0,T]$.
\end{lemma}

\begin{lemma}[See \cite{Had1}]\label{lem 3.3}
Let $0<\alpha <\frac{2T}{\eta ^{2}}$, $\ 0\leq\beta <%
\frac{2T-\alpha \eta ^{2}}{\alpha \eta ^{2}-2\eta +2T}$. If $y\in
C([0,T],[0,\infty ))$, then the unique solution $u$ of the problem
\eqref{eq-14}, \eqref{eq-15} satisfies
\begin{equation} \label{eq-16}
\min_{t\in [\eta,T]}u(t)\geq \gamma \|u\|,\ \|u\|=\max_{t\in
[0,T]}|u(t)|,
\end{equation}
where
\begin{equation} \label{eq-17}
\gamma:=\min\left\{\frac{\eta}{T},
\frac{\alpha(\beta+1)\eta^{2}}{2T},
\frac{\alpha(\beta+1)\eta(T-\eta)}{2T-\alpha(\beta+1)\eta^{2}}\right\}\in\left(0,1\right).
\end{equation}
\end{lemma}

\section{Existence of triple solutions}

In this section, we discuss the multiplicity of positive solutions for the general
boundary-value problem \eqref{eq-4}, \eqref{eq-5}

In the following, we denote

\begin{gather}
\Lambda:=(2T-\alpha\eta^{2})-\beta(\alpha\eta^{2}-2\eta+2T), \label{eq-18}\\
m:=\Big(\frac{T^{2}(2T(\beta+1)+\beta\eta(\alpha
\eta+2)+\alpha\beta T^{2})}{2\Lambda}\Big) ^{-1},  \label{eq-19}\\
\delta :=\min\Big\{\frac{\beta\eta(T-\eta)^{2}}{\Lambda},
 \frac{\alpha\eta^{2}(1+\beta)(T-\eta)^{2}}{2\Lambda} \Big\}. \label{eq-20}
\end{gather}

Using Theorem \ref{theo 1.1}, we established the following existence theorem for the boundary-value
problem \eqref{eq-4}, \eqref{eq-5}.

\begin{theorem}\label{threesolution}
Assume {\rm (H1)} and {\rm (H2)} hold.
Suppose there exists constants $0<a<b<b/\gamma\le c$ such that
\begin{enumerate}
\item[(D1)] $f(t, u)<ma$ for  $t\in[0, T]$, $u\in[0, a]$;
\item[(D2)] $f(t, u)\ge\frac b\delta$  for $t\in[\eta, T]$,
$u\in[b, \frac b\gamma]$;
\item[(D3)] $f(t, u)\le mc$ for
$t\in[0, T]$, $u\in[0,c]$,
\end{enumerate}
where $\gamma, m, \delta$ are as defined in \eqref{eq-17},
\eqref{eq-19} and \eqref{eq-20}, respectively. Then the
boundary-value problem \eqref{eq-4}-\eqref{eq-5} has at least
three positive solutions $u_1, u_2$ and $u_3$ satisfying
$$
\|u_1\|<a,\quad\min_{t\in[0, T]}u_2(t)>b,\quad a<\|u_3\|\quad\text{with }
\min_{t\in[0, T]}u_3(t)<b.
$$
\end{theorem}

\begin{proof}

Let $E = C([0, T], \mathbb{R})$ be endowed with the maximum norm, $\|u\|=\max_{t\in[0, T]} u(t)$, define the cone $P\subset C([0, T], \mathbb{R})$ by
\begin{equation}
P=\{u\in C([0, T], \mathbb{R}) : u \text{ concave down and }
u(t)\ge 0 \text{ on } [0, T]\}.  \label{eq-21}
\end{equation}
Let $\psi:P\to [0,\infty)$ be defined by
\begin{equation}
\psi(u)=\min_{t\in[0, T]} u(t),\quad u\in P. \label{eq-22}
\end{equation}
then $\psi$ is a nonnegative continuous concave functional and
$\psi(u)\le\|u\|, u\in P$.

Define the operator $A: P\to C([0, T], \mathbb{R})$ by

\begin{eqnarray*}
Au(t)&=&-\frac{\beta (2T-\alpha \eta ^{2})-2\beta (1-\alpha \eta
)t}{\Lambda}\int_{0}^{\eta }(\eta -s)f(s, u(s))ds \\
&&-\frac{\alpha \beta \eta -\alpha (\beta -1)t}{\Lambda}\int_{0}^{\eta }(\eta -s)^{2}f(s, u(s))ds \\
&&-\frac{2(\beta-1)t-2\beta \eta }{\Lambda}\int_{0}^{T}(T-s)f(s, u(s))ds \\
&&-\int_{0}^{t}(t-s)f(s, u(s))ds.
\end{eqnarray*}

Then the fixed points of $A$ just are the solutions of the
boundary-value problem \eqref{eq-4}-\eqref{eq-5} from Lemma
\ref{lem 3.1}. Since $(Au)^{\prime \prime }(t)=-f(t,  u(t))$ for $t\in
(0, T)$, together with {\rm (H1)} and Lemma \ref{lem 3.2}, we see that
$Au(t)\ge 0,\ t\in [0, T]$ and $(Au)^{\prime \prime }(t)\le 0,\
t\in(0, T)$. Thus $A:P\to P$. Moreover, $A$ is completely
continuous.

We now show that all the conditions of Theorem \ref{theo 1.1} are satisfied. From
\eqref{eq-22}, we know that $\psi(u)\le\|u\|$, for all $u\in P$.

Now if $u\in\overline {P_c}$, then $0\le u\le c$, together with {\rm (D3)},
we find $\forall\ t\in[0, T]$,

\begin{eqnarray*}
Au(t)&\leq&\frac{2\beta (1-\alpha \eta
)t-\beta (2T-\alpha \eta ^{2})}{\Lambda}\int_{0}^{\eta }(\eta -s)f(s, u(s))ds \\
&&+\frac{\alpha (\beta -1)t-\alpha \beta \eta}{\Lambda}\int_{0}^{\eta }(\eta -s)^{2}f(s, u(s))ds \\
&&+\frac{2\beta \eta-2(\beta-1)t }{\Lambda}\int_{0}^{T}(T-s)f(s, u(s))ds \\
&\leq&\frac{2\beta T+\alpha\beta \eta ^{2}}{\Lambda}\int_{0}^{\eta }(\eta -s)f(s, u(s))ds +\frac{\alpha\beta T}{\Lambda}\int_{0}^{\eta }(\eta -s)^{2}f(s, u(s))ds \\
&&+\frac{2\beta \eta+2T}{\Lambda}\int_{0}^{T}(T-s)f(s, u(s))ds \\
&\leq&\frac{2T(\beta+1)+\beta\eta(\alpha\eta+2)}{\Lambda}
\int_{0}^{T}(T-s)f(s, u(s))ds \\
&&+\frac{\alpha\beta T}{\Lambda}\int_{0}^{\eta}(\eta -s)^{2}f(s, u(s))ds \\
&\leq&\frac{2T(\beta+1)+\beta\eta(\alpha\eta+2)}{\Lambda}
\int_{0}^{T}(T-s)f(s, u(s))ds \\
&&+\frac{\alpha\beta T}{\Lambda}\int_{0}^{T}T(T-s)f(s, u(s))ds \\
&=&\frac{2(\beta+1)+T^{-1}\beta\eta(\alpha\eta+2)+\alpha\beta
T}{\Lambda}\int_{0}^{T}T(T-s)f(s, u(s))ds \\
&\leq& mc \frac{2(\beta+1)+T^{-1}\beta\eta(\alpha\eta+2)+\alpha\beta
T}{\Lambda}\int_{0}^{T}T(T-s)ds\\
&=& mc \frac{T^{2}(2T(\beta+1)+\beta\eta(\alpha
\eta+2)+\alpha\beta T^{2})}{2\Lambda}\\
&=& c
\end{eqnarray*}

Thus, $A:\overline {P_c}\to \overline {P_c}$.

By {\rm (D1)} and the argument above, we can get that
$A:\overline{P_a}\to P_a$. So, $\|Au\|<a$ for $\|u\|\le a$, the condition {\rm (C2)}
of Theorem \ref{theo 1.1} holds.

Consider the condition {\rm (C1)} of Theorem \ref{theo 1.1} now. Since
$\psi(b/\gamma)=b/\gamma>b$, let $d=b/\gamma$, then $\{u\in
P(\psi, b, d):\psi(u)>b\}\neq\emptyset$. For $u\in
P(\psi, b, d)$, we have $b\le u(t)\le b/\gamma,\ t\in[0, T]$.
Combining with {\rm (D2)}, we get
$$
f(t, u)\ge\frac b\delta,\quad t\in[\eta, T].
$$
Since $u\in P(\psi, b, d)$, then there are two cases, (i) $\psi(Au)(t)=Au(0)$ and (ii) $\psi(Au)(t)=Au(T)$.
In case (i), we have

\begin{eqnarray*}
\psi(Au)(t)&=& Au(0)\\
&=&-\frac{\beta (2T-\alpha \eta ^{2})}{\Lambda}\int_{0}^{\eta }(\eta -s)f(s, u(s))ds\\
&&-\frac{\alpha \beta \eta }{\Lambda}\int_{0}^{\eta }(\eta -s)^{2}f(s, u(s))ds +\frac{2\beta \eta }{\Lambda}\int_{0}^{T}(T-s)f(s, u(s))ds \\
&=&\frac{2\beta\eta}{\Lambda}\int_{0}^{T}(T-s)f(s, u(s))ds-\frac{\beta\eta(2T-\alpha\eta^{2)} }{\Lambda}\int_{0}^{\eta }f(s, u(s))ds \\
&&+\frac{\beta(2T-\alpha\eta^{2)}}{\Lambda}\int_{0}^{\eta}sf(s, u(s))ds\\
&&-\frac{\alpha\beta\eta }{\Lambda}\int_{0}^{\eta}(\eta^{2}-2\eta s+s^{2})f(s, u(s))ds\\
&=&\frac{2\beta\eta}{\Lambda}\int_{0}^{T}(T-s)f(s, u(s))ds-\frac{\alpha\beta\eta}{\Lambda}\int_{0}^{\eta }s^{2}f(s, u(s))ds \\
&&-\frac{2T\beta\eta}{\Lambda}\int_{0}^{\eta}f(s, u(s))ds +\frac{2T\beta+\alpha\eta^{2}\beta}{\Lambda}\int_{0}^{\eta}sf(s, u(s))ds \\
&=&\frac{2T\beta\eta}{\Lambda}\int_{\eta}^{T}f(s, u(s))ds-\frac{2\beta\eta}{\Lambda}\int_{0}^{T }sf(s, u(s))ds \\
&&-\frac{\alpha\beta\eta}{\Lambda}\int_{0}^{\eta}s^{2}f(s, u(s))ds +\frac{2T\beta+\alpha\eta^{2}\beta}{\Lambda}\int_{0}^{\eta}sf(s, u(s))ds \\
&=&\frac{2\beta\eta}{\Lambda}\int_{\eta}^{T}(T-s)f(s, u(s))ds+\frac{2\beta(T-\eta)+\alpha\eta^{2}\beta}{\Lambda}\int_{0}^{\eta }sf(s, u(s))ds \\
&&-\frac{\alpha\beta\eta}{\Lambda}\int_{0}^{\eta}s^{2}f(s, u(s))ds\\
&>& \frac{2\beta\eta}{\Lambda}\int_{\eta}^{T}(T-s)f(s, u(s))ds+\frac{\alpha\eta^{2}\beta}{\Lambda}\int_{0}^{\eta}sf(s, u(s))ds \\
&&-\frac{\alpha\beta\eta}{\Lambda}\int_{0}^{\eta}s^{2}f(s, u(s))ds\\
&=& \frac{2\beta\eta}{\Lambda}\int_{\eta}^{T}(T-s)f(s, u(s))ds+\frac{\alpha\beta\eta}{\Lambda}\int_{0}^{\eta}s(\eta-s)f(s, u(s))ds\\
&>& \frac{2\beta\eta}{\Lambda}\int_{\eta}^{T}(T-s)f(s, u(s))ds\\
&\geq& \frac{b}{\delta}\frac{2\beta\eta}{\Lambda}\int_{\eta}^{T}(T-s)ds\\
&=& \frac{b}{\delta}\frac{\beta\eta(T-\eta)^{2}}{\Lambda}\\
&\geq& b.
\end{eqnarray*}

In case (ii), we have
\begin{eqnarray*}
\psi(Au)(t)&=& Au(T)\\
&=&-\frac{\beta (2T-\alpha \eta ^{2})-2\beta (1-\alpha \eta
)T}{\Lambda}\int_{0}^{\eta }(\eta -s)f(s, u(s))ds \\
&&-\frac{\alpha \beta \eta -\alpha (\beta -1)T}{\Lambda}\int_{0}^{\eta}(\eta -s)^{2}f(s, u(s))ds \\
&&-\frac{2(\beta-1)T-2\beta \eta }{\Lambda}\int_{0}^{T}(T-s)f(s, u(s))ds \\
&&-\int_{0}^{T}(T-s)f(s, u(s))ds\\
&=&\frac{\alpha\beta\eta(\eta-2T)}{\Lambda}\int_{0}^{\eta}(\eta -s)f(s, u(s))ds \\
&&+\frac{\alpha(\beta-1)T-\alpha\beta\eta}{\Lambda}\int_{0}^{\eta}(\eta -s)^{2}f(s, u(s))ds \\
&&+\frac{\alpha\eta^{2}(\beta+1)}{\Lambda}\int_{0}^{T}(T-s)f(s, u(s))ds \\
&=&\frac{\alpha\eta^{2}(\beta+1)}{\Lambda}\int_{0}^{T}(T-s)f(s, u(s))ds
-\frac{\alpha\eta^{2}T(\beta+1)}{\Lambda}\int_{0}^{\eta}f(s, u(s))ds \\
&&+\frac{\alpha\eta(\beta\eta+2T)}{\Lambda}\int_{0}^{\eta}sf(s, u(s))ds\\
&&+\frac{\alpha(\beta-1)T-\alpha\beta\eta}{\Lambda}\int_{0}^{\eta}s^{2}f(s, u(s))ds\\
&=&\frac{\alpha\eta^{2}T(\beta+1)}{\Lambda}\int_{\eta}^{T}f(s, u(s))ds
-\frac{\alpha\eta^{2}(\beta+1)}{\Lambda}\int_{0}^{\eta}sf(s, u(s))ds \\
&&-\frac{\alpha\eta^{2}(\beta+1)}{\Lambda}\int_{\eta}^{T}sf(s, u(s))ds
+\frac{\alpha\eta(\beta\eta+2T)}{\Lambda}\int_{0}^{\eta}sf(s, u(s))ds\\
&&+\frac{\alpha(\beta-1)T-\alpha\beta\eta}{\Lambda}\int_{0}^{\eta}s^{2}f(s, u(s))ds\\
&=&\frac{\alpha\eta^{2}(\beta+1)}{\Lambda}\int_{\eta}^{T}(T-s)f(s, u(s))ds
+\frac{\alpha\eta(2T-\eta)}{\Lambda}\int_{0}^{\eta}sf(s, u(s))ds \\
&&+\frac{\alpha\beta(T-\eta)-\alpha T}{\Lambda}\int_{0}^{\eta}s^{2}f(s, u(s))ds\\
&>&\frac{\alpha\eta^{2}(\beta+1)}{\Lambda}\int_{\eta}^{T}(T-s)f(s, u(s))ds+
\frac{\alpha\eta T}{\Lambda}\int_{0}^{\eta}sf(s,u(s))ds\\
&&-\frac{\alpha T}{\Lambda}\int_{0}^{\eta}s^{2}f(s, u(s))ds\\
&=&\frac{\alpha\eta^{2}(\beta+1)}{\Lambda}\int_{\eta}^{T}(T-s)f(s, u(s))ds+
\frac{\alpha T}{\Lambda}\int_{0}^{\eta}s(\eta-s)f(s, u(s))ds\\
&>&\frac{\alpha\eta^{2}(\beta+1)}{\Lambda}\int_{\eta}^{T}(T-s)f(s, u(s))ds\\
&\geq&\frac{b}{\delta}\frac{\alpha\eta^{2}(\beta+1)}{\Lambda}\int_{\eta}^{T}(T-s)ds\\
&=&\frac{b}{\delta}\frac{\alpha\eta^{2}(\beta+1)(T-\eta)^{2}}{2\Lambda}\\
&\geq& b.
\end{eqnarray*}
So,  $\psi(Au)>b$ ; $ \forall$ $u\in P(\psi, b, b/\gamma)$.

For the condition {\rm (C3)} of the Theorem \ref{theo 1.1}, we can verify
it easily under our assumptions using Lemma \ref{lem 3.3}. Here
$$
\psi(Au)=\min_{t\in [0, T]} Au(t) \ge \gamma\|Au\|
        > \gamma \frac b\gamma = b
$$
as long as $u\in P(\psi, b, c)$ with $\|Au\|>b/\gamma$.

Since all conditions of Theorem \ref{theo 1.1} are satisfied. Then problem \eqref{eq-4}-\eqref{eq-5} has at least three positive solutions
$u_1$,\ $u_2$,\ $u_3$ with
$$
\|u_1\|<a,\quad\psi(u_2)>b,\quad a<\|u_3\|\quad\text{with }\psi(u_3)<b.
$$
\end{proof}

\section{Some examples}
In this section, in order to illustrate our result, we consider some examples.

\begin{example}
Consider the boundary value problem

\begin{equation}\label{eq-37}
{u^{\prime \prime }}(t)+\frac{40u^{2}}{u^{2}+1}=0, \  \ 0<t<1,
\end{equation}

\begin{equation}\label{eq-38}
u(0)=\frac{1}{2}u(\frac{1}{3}), \  \ u(1)=3 \int_{0}^{\frac{1}{3}}u(s)ds.
\end{equation}

Set $\beta=1/2$, $\alpha=3$, $\eta=1/3$, $T=1$, and

\[f(t, u)=f(u)=\frac{40u^{2}}{u^{2}+1}, \ \ u\geq0.\]

It is clear that $f(.)$ is continuous and increasing on $[0, \infty)$. We can also seen that

\[0<\alpha=3<18=\frac{2T}{\eta^{2}}, \ \ 0<\beta=\frac{1}{2}<1=\frac{2T-\alpha \eta ^{2}}{\alpha \eta ^{2}-2\eta
+2T}.\]

Now we check that {\rm (D1)}, {\rm (D2)} and {\rm (D3)} of Theorem \ref{threesolution} are satisfied.
By \eqref{eq-17}, \eqref{eq-19}, \eqref{eq-20}, we get $\gamma=1/4$, $m=1/3$, $\delta=4/45$. Let $c=124$, we have
\[f(u)\leq40<mc=\frac{124}{3}\approx41,33, \ u\in[0, c], \]

from $\lim_{u\rightarrow\infty}f(u)=40$, so that {\rm (D3)} is met. Note that $f(2)=32$, when we set $b=2$,

\[f(u)\geq\frac{b}{\delta}=22.5, \ u\in[b, 4b],\]
holds. It means that {\rm (D2)}is satisfied. To verify {\rm (D1)}, as $f(\frac{1}{120})=\frac{40}{14401}$, we take $a=\frac{1}{120}$, then

\[f(u)<ma=\frac{1}{360}, \ u\in[0, a],\]
and {\rm (D1)} holds.
Summing up, there exists constants $a=1/120$, $b=2$, $c=124$ satisfying

\[0<a<b<\frac{b}{\gamma}\leq c,\]
such that {\rm (D1)}, {\rm (D2)} and {\rm (D3)} of Theorem \ref{threesolution} hold. So the boundary-value problem \eqref{eq-37}-\eqref{eq-38}  has at least three positive solutions $u_{1}$, $u_{2}$ and $u_{3}$ satisfying

\[\|u_{1}\|<\frac{1}{120}, \ \  \min_{t\in[0, T]}u_{2}(t)>2, \ \ \frac{1}{120}<\|u_{3}\| \ \ \text{with} \ \min_{t\in[0, T]}u_{3}(t)<2 .\]

\end{example}

\begin{example}
Consider the boundary value problem

\begin{equation}\label{eq-39}
{u^{\prime \prime }}(t)+f(t, u)=0, \  \ 0<t<1,
\end{equation}

\begin{equation}\label{eq-40}
u(0)=u(\frac{1}{2}), \  \ u(1)= \int_{0}^{\frac{1}{2}}u(s)ds.
\end{equation}

Set $\beta=1$, $\alpha=1$, $\eta=1/2$, $T=1$, $f(t, u)=e^{-t}h(u)$ where

 \begin{equation}\label{green}
h(u) = \begin{cases} \frac{2}{25}u  &  0\leq u\leq 1  \\
\frac{2173}{75}u-\frac{2167}{75} &  1\leq u\leq 4 \\
 87   &   4\leq u\leq 544 \\
\frac{87}{544}u   &  544\leq u\leq 546 \\
\frac{39(3u+189)}{u+270} &  u\geq 546.
\end{cases}
\end{equation}

By \eqref{eq-17}, \eqref{eq-19}, \eqref{eq-20} and after a simple calculation, we get $\gamma=1/4$, $m=4/25$, $\delta=1/8$.

We choose $a = 1/4$, $b = 4$, and $c = 544$; consequently,

\begin{eqnarray*}
f(t, u)&=&e^{-t}\frac{2}{25}u\leq \frac{2}{25}u<\frac{4}{25}\times\frac{1}{4}=ma, \ \  0\leq t\leq 1, \ \ 0\leq u\leq 1/4,\\
f(t, u)&=&e^{-t}87\geq \frac{87}{e}>32=\frac{b}{\delta}, \ \  1/2\leq t\leq 1, \ \ 4\leq u\leq 16,\\
f(t, u)&=&e^{-t}h(u)\leq 87<\frac{4}{25}\times544=mc, \ \  0\leq t\leq 1, \ \ 0\leq u\leq 544.
\end{eqnarray*}

That is to say, all the conditions of Theorem \ref{threesolution} are satisfied. Then problem
\eqref{eq-39}, \eqref{eq-40} has at least three positive solutions $u_{1}$ $u_{2}$, and $u_{3}$ satisfying

$$
\|u_1\|<\frac{1}{4},\quad\psi(u_2)>4,\quad \|u_3\|>\frac{1}{4} \ \ \text{with}\ \ \psi(u_3)<4.
$$

\end{example}

\end{document}